\documentclass{lmcs}
\pdfoutput=1

\usepackage{lastpage}
\lmcsdoi{15}{2}{14}
\lmcsheading{}{\pageref{LastPage}}{}{}%
{Jun.~29,~2018}{May~23,~2019}{}


\usepackage[utf8]{inputenc}
\usepackage{amsmath,amssymb}
\usepackage{hyperref}
\usepackage{xypic}

\newcommand{\Ktwo}{\mathcal{K}_2}
\newcommand{\RT}{\mathsf{RT}(\Ktwo)}

\newcommand{\RR}{\mathbb{R}} 
\newcommand{\NN}{\mathbb{N}} 
\newcommand{\BB}{\NN^\NN} 
\newcommand{\INN}{\mathsf{N}} 
\newcommand{\IZZ}{\mathsf{Z}} 
\newcommand{\IRR}{\mathsf{R}} 
\newcommand{\IBB}{\mathsf{B}} 
\newcommand{\ISS}{\mathsf{\Sigma}} 

\newcommand{\nablaTwo}{\nabla\mathsf{2}}

\newcommand{\INNx}{\INN^{*}} 
\newcommand{\prefix}[2]{\overline{#1}(#2)} 
\newcommand{\length}[1]{|#1|} 
\newcommand{\padz}[1]{#10^\omega} 

\newcommand{\pow}[1]{\mathcal{P}(#1)} 

\newcommand{\onto}{\twoheadrightarrow}
\newcommand{\into}{\rightarrowtail}

\newcommand{\some}[1]{\exists #1 \,.\,}
\newcommand{\all}[1]{\forall #1 \,.\,}

\newcommand{\lthen}{\Rightarrow}
\newcommand{\liff}{\Leftrightarrow}

\newcommand{\set}[1]{\{#1\}} 
\newcommand{\such}{\mid} 

\begin{document}

\title{Every metric space is separable in function realizability}

\author{Andrej Bauer}
\address{Andrej Bauer\\
Faculty of Mathematics and Physics\\
University of Ljubljana\\
Jadranska 19\\
1000 Ljubljana\\
Slovenia}
\email{Andrej.Bauer@andrej.com}
\thanks{The first author acknowledges that this material is based upon work
  supported by the Air Force Office of Scientific Research under award number
  FA9550-17-1-0326.}

\author{Andrew Swan}
\address{Andrew Swan\\
Institute for Logic, Language and Computation\\
University of Amsterdam\\
Science Park 107\\
1098 XG Amsterdam\\
Netherlands}
\email{wakelin.swan@gmail.com}

\begin{abstract}
  We first show that in the function realizability topos $\RT$ every metric space is
  separable, and every object with decidable equality is countable. More generally,
  working with synthetic topology, every $T_0$-space is separable and every discrete space
  is countable.
  It follows that intuitionistic logic does not show the existence of a non-separable
  metric space, or an uncountable set with decidable equality, even if we assume
  principles that are validated by function realizability, such as Dependent and Function
  choice, Markov's principle, and Brouwer's continuity and fan principles.
\end{abstract}

\maketitle

Are there any uncountable sets with decidable equality in constructive mathematics, or at
least non-separable metric spaces? We put these questions to rest by showing that in the
function realizability topos all metric spaces are separable, and consequently all sets
with decidable equality countable. Therefore, intuitionistic logic does not show existence
of non-separable metric spaces, even if we assume principles that are validated by
function realizability, among which are the Dependent and Function choice, Markov's
principle, and Brouwer's continuity and fan principles.

\section{Function realizability}

We shall work with the realizability topos~$\RT$, see~\cite[\S4.3]{oosten08:_realiz},
which is based on Kleene's function realizability~\cite{KleeneSC:fouim}. We carry out the
bulk of the argument in the internal language of the topos, which is intuitionistic logic
with several extra principles, cf.~Proposition~\ref{prop:rt-principles}.

We write $\INN$ and $\IRR$ for the objects of the natural numbers and the real numbers,
respectively. The Baire space is the object $\IBB = \INN^\INN$ of infinite number
sequences. It is metrized by the metric $u : \IBB \times \IBB \to \IRR$ defined by
\begin{equation*}
  u(\alpha, \beta) =
  \lim_{n \to \infty}
  2^{-\min \set{k \leq n \such k = n \lor \alpha_k \neq \beta_k}}.
\end{equation*}
If the first index at which $\alpha$ and $\beta$ differ is $k$, then
$u(\alpha, \beta) = 2^{-k}$.

\begin{prop}
  \label{prop:rt-principles}
  The realizability topos~$\RT$ validates the following principles:
  \begin{enumerate}
  \item \emph{Countable choice:} a total relation on $\INN$ has a choice map.
  \item \emph{Extended function choice:} if $S \subseteq \IBB$ is $\lnot\lnot$-stable then every
    total relation on $S$ has a choice map.
  \item \emph{Extended continuity principle:} if $S \subseteq \IBB$ is $\lnot\lnot$-stable
    then every map $S \to \IBB$ is continuous.
  \item \emph{Excluded middle for predicates on~$\INN$}: if $\phi(n)$ is a formula whose
    only parameter is $n \in \INN$, then $\all{n \in \INN} \phi(n) \lor \lnot \phi(n)$.
  \end{enumerate}
\end{prop}

\begin{proof}
  \parbox{0pt}{}
  \begin{enumerate}
  \item Every realizability topos validates Countable choice, and
    \cite[Prop.~4.3.2]{oosten08:_realiz} does so specifically for $\RT$.

  \item Recall that a subobject $S \subseteq \IBB$ is $\lnot\lnot$-stable when
    $\lnot\lnot (\alpha \in S)$ implies $\alpha \in S$ for all $\alpha \in \IBB$. The
    realizability relation on such an $S$ is inherited by that of $\IBB$, i.e., the
    elements of~$S$ are realized by Kleene's associates. The argument proceeds the same
    way as \cite[Prop.~4.3.2]{oosten08:_realiz}, which shows that choice holds in the case
    $S = \IBB$.

  \item Once again, if $S \subseteq \IBB$ is $\lnot\lnot$-stable, then maps $S \to \IBB$
    are realized by Kleene's associates, just like maps from $\IBB \to \IBB$. The argument
    proceeds the same way as continuity of maps $\IBB \to \IBB$
    in~\cite[Prop.~4.3.4]{oosten08:_realiz}.

  \item Let us first show that, for a formula $\psi(n)$ whose only parameter is
    $n \in \INN$, the sentence
    \begin{equation}
      \label{eq:n-lem-eq}
      \all{n \in \INN} \lnot\lnot \psi(n) \lthen \psi(n)
    \end{equation}
    is realized. The formula $\psi(n)$ is interpreted as a subobject of~$\INN$, which is
    represented by a map $f : \NN \to \pow{\BB}$. Using a bit of (external) classical
    logic and Countable choice we obtain a map $c : \NN \to \BB$ such that
    \begin{equation*}
      \all{n \in \NN}
      f(n) \neq \emptyset
      \lthen
      c(n) \in f(n),
    \end{equation*}
    which says that $c$ can be used to build a realizer for~\eqref{eq:n-lem-eq}. 

    Now, given a formula $\phi(n)$ whose only parameter is $n \in \INN$, take $\psi(n)$ to
    be $\phi(n) \lor \lnot \phi(n)$. Because $\lnot\lnot (\phi(n) \lor \lnot\phi(n))$
    holds, \eqref{eq:n-lem-eq} reduces to the desired statement
    $\all{n \in \INN} \phi(n) \lor \lnot \phi(n)$. \qedhere
\end{enumerate}
\end{proof}

Note that the last part of the previous proposition does \emph{not} state the
validity of the internal statement
\begin{equation*}
  \all{\phi \in \Omega^\INN} \all{n \in \INN} \phi(n) \lor \lnot\phi(n).
\end{equation*}
Indeed, such a statement cannot be valid in any realizability topos because it implies
excluded middle (given $p \in \Omega$, consider $\phi(n) \equiv p$). Rather, we have a
\emph{schema} which holds for each formula~$\phi(n)$.

\section{Metric spaces in function realizability}

Henceforth we argue in the internal language of~$\RT$. A minor but common
complication arises because the internal language does not allow quantification
over all objects of the topos. Thus, when we make a statement $\phi_{(X,d)}$
about all metric spaces $(X,d)$ in the topos, this is to be understood
schematically: given any object $X$ and morphism $d : X \times X \to \IRR$, if
the topos validates ``$d$ is a metric'' then it also validates $\phi_{(X,d)}$.

A consequence of Countable choice is that the object of reals~$\IRR$
is a continuous image of $\IBB$, for instance by composing any bijection
$\IBB \cong \IZZ^\INN$ with the surjection
$\IZZ^\INN \onto \IRR$ defined by taking
$\alpha \in \IZZ^\INN$ to
$\alpha_0 + \sum_{k = 0}^\infty 2^{-k} \cdot \max(-1, \min(1,
\alpha_{k+1}))$.

We follow the convention that a set $X$ is \emph{countable} if there
exists a surjection $\INN \onto X + 1$, and we refer to such a
surjection as an \emph{enumeration} of $X$. Note in particular that
the empty set is countable, and that if $X$ is inhabited, then it is
countable if and only if there is a surjection $\INN \onto X$.

Recall that a metric space $(X, d)$ is \emph{separable} if there exists a countable subset
$D \subseteq X$ such that, for all $x \in X$ and $k \in \INN$, there exists $y \in D$ such
that $d(x, y) \leq 2^{-k}$.

\begin{prop}
  \label{prop:separable-quotient}
  Suppose $(X, d_X)$ is a separable metric space and $q : X \onto Y$ is a surjection onto
  a metric space $(Y, d_Y)$ such that $d_Y \circ (q \times q) : X \times X \to \IRR$ is
  continuous with respect to the metric $d_X$. Then $Y$ is separable.
\end{prop}

\begin{proof}
  Let $D \subseteq X$ be a countable dense subset of~$X$. We claim that its image $q(D)$
  is dense in~$Y$. Consider any $y \in Y$ and $k \in \INN$. There is $x \in X$ such that
  $q(x) = y$. Because $d_Y \circ (q \times q)$ is continuous at~$(x, x)$ there exists
  $m \in \INN$ such that, for all $z \in X$, if $d_X(x, z) \leq 2^{-m}$ then
  \begin{equation*}
    d_Y(q(x), q(z)) =
    |d_Y(q(x), q(z)) - d_Y(q(x), q(x))| \leq 2^{-k}.
  \end{equation*}
  Since $D$ is dense there exists $z \in D$ such that
  $d_X(x, z) \leq 2^{-m}$, hence $d_Y(y, q(z)) = d_Y(q(x), q(z)) \leq 2^{-k}$.
\end{proof}

\begin{prop}
  \label{prop:subB-separable}
  For any subobject $S \into \IBB$, the topos $\RT$ validates the statement that $(S, u)$
  is a separable metric space.
\end{prop}

\begin{proof}
  Let $\INNx$ be the object of finite sequences of numbers. We write $\length{a}$ for the
  length~$n$ of a sequence $a = (a_0, \ldots, a_{n-1})$. For $\alpha \in \IBB$ and
  $k \in \INN$, let $\prefix{\alpha}{k} = (\alpha_0, \ldots, \alpha_{k-1})$ be the prefix
  of $\alpha$ of length~$k$. Given a finite sequence $a = (a_0, \ldots, a_{n-1})$, let
  $\padz{a}$ be its padding by zeroes,
  \begin{equation*}
    \padz{a} = (a_0, \ldots, a_{n-1}, 0, 0, 0, \ldots).
  \end{equation*}
  Notice that $u(\padz{\prefix{\alpha}{k}}, \alpha) \leq 2^{-k}$ for all
  $\alpha \in \IBB$ and $k \in \INN$.

  Because $\INNx$ is isomorphic to $\INN$, we may apply Excluded middle for predicates
  on~$\INN$ to establish
  \begin{equation*}
    \all{a \in \INNx} (\some{\alpha \in S} u(\padz{a}, \alpha) \leq 2^{-\length{a}}) \lor
    \lnot(\some{\alpha \in S} u(\padz{a}, \alpha) \leq 2^{-\length{a}}).
  \end{equation*}
  By Countable choice there is a map $c : \INNx \to S + 1$ such that, for all
  $a \in \INNx$, if there exists $\alpha \in S$ with
  $u(\padz{a}, \alpha) \leq 2^{-\length{a}}$ then $c(a) \in S$ and
  $u(\padz{a}, c(a)) \leq 2^{-\length{a}}$. We claim that~$c$ enumerates a dense sequence
  in~$S$. To see this, consider any $\alpha \in S$ and $k \in \INN$. Because
  $u(\padz{\prefix{\alpha}{k+1}}, \alpha) \leq 2^{-k-1}$, we have
  $c(\prefix{\alpha}{k+1}) \in S$ and therefore
  \begin{multline*}
    u(\alpha, c(\prefix{a}{k+1}))
    \leq \\
    u(\alpha, \padz{\prefix{a}{k+1}}) + u(\padz{\prefix{a}{k+1}}, c(\prefix{a}{k+1}))
    \leq 2^{-k-1} + 2^{-k-1} = 2^{-k}.
    \qedhere
  \end{multline*}
\end{proof}

\begin{prop}
  \label{prop:subqoutient-separable}
  A metric space is separable if its carrier is the quotient of a $\lnot\lnot$-stable
  subobject of~$\IBB$.
\end{prop}

\begin{proof}
  Suppose $(X, d)$ is a metric space such that there exist a $\lnot\lnot$-stable
  $S \subseteq \IBB$ and a surjection $q : S \onto X$.
  By Proposition~\ref{prop:subB-separable}, the space $(S, u)$ is separable. We may apply
  Proposition~\ref{prop:separable-quotient}, provided that
  $d \circ (q \times q) : S \times S \to \IRR$ is continuous with respect to~$u$. By
  the Extended function choice, $d \circ (q \times q)$ factors through a continuous surjection
  $\IBB \onto \IRR$ as
  \begin{equation*}
    \xymatrix{
      {S \times S}
      \ar[dr]_{d \circ (q \times q)}
      \ar[rr]^{f}
      & &
      {\IBB}
      \ar@{->>}[ld]
      \\
      &
      {\IRR}
    }
  \end{equation*}
  By the Extended Continuity principle the map~$f$ is
  continuous, hence $d \circ (q \times q)$ is continuous, too.
\end{proof}

Before proceeding we review the notion of a \emph{modest}
object~\cite[Def.~3.2.23]{oosten08:_realiz}: $X$ is modest when it has $\lnot\lnot$-stable
equality and is orthogonal to the object
$\nablaTwo$~\cite[Prop.~3.2.22]{oosten08:_realiz}.\footnote{The results
  of~\cite[\S3.2]{oosten08:_realiz} refer specifically to the effective topos, but are
  easily adapted to any realizability topos, as long as one replaces $\INN$ with the
  underlying partial combinatory algebra, especially
  in~\cite[Def.~3.2.17]{oosten08:_realiz}.}
The modest objects are, up to isomorphism, the quotients by $\lnot\lnot$-stable
equivalence relations of $\lnot\lnot$-stable subobjects of the underlying
partial combinatory algebra, which in our case is the Baire space~$\IBB$. The
powers and the subobjects of a modest set are modest (see the remark
after~\cite[Def.~3.2.23]{oosten08:_realiz} about applicability
of~\cite[Prop.~3.2.19]{oosten08:_realiz} to modest objects).

\begin{thm}
  \label{thm:metric-separable}
  Every metric space in $\RT$ is separable.
\end{thm}

\begin{proof}
  Consider a metric space $(X, d)$ in $\RT$. The object of reals~$\IRR$ is modest because
  it has $\lnot\lnot$-stable equality and is a quotient of~$\IBB$. Its power $\IRR^X$ is
  modest, and because the transpose of the metric $\widetilde{d} : X \into \IRR^X$
  embeds~$X$ into $\IRR^X$, the carrier $X$ is modest, therefore a quotient of a
  $\lnot\lnot$-stable subobject of~$\IBB$. We may apply
  Proposition~\ref{prop:subqoutient-separable}.
\end{proof}

\begin{thm}
  \label{thm:decidable-countable}
  In $\RT$ every object with decidable equality is countable.
\end{thm}

\begin{proof}
  The precise statement is: for any object $X$ in $\RT$, $\RT$ validates the statement
  ``if $X$ has decidable equality then $X$ is countable''.

  We argue internally. If equality on $X$ is decidable then we may define the discrete
  metric $d : X \times X \to \RR$ by
  \begin{equation*}
    d(x,y) =
    \begin{cases}
      0 & \text{if $x = y$,}\\
      1 & \text{otherwise}
    \end{cases}
  \end{equation*}
  Because $(X, d)$ is separable it contains a countable dense subset $D \subseteq X$. But
  then $D = X$, because for any $x \in X$, there is $y \in D$ such that $d(x, y) < 1/2$,
  which implies $x = y$.
\end{proof}

The upshot of Theorems~\ref{thm:metric-separable}
and~\ref{thm:decidable-countable} is that in constructive mathematics a
non-separable metric space cannot be constructed, and neither can an uncountable
set with decidable equality, for such constructions could be interpreted in
$\RT$ to give counterexamples to the theorems.

The theorems fail for \emph{families} of objects. For instance, while for any specific
subobject $S \into \INN$ the statement ``$S$ is countable'' is valid, the internal
statement ``every subobject of~$\INN$ is countable'' is invalid. It is a variation of
\emph{Kripke's schema}~\cite{troelstra69:_princ} which together with Markov's principle
implies Excluded middle, as follows.

\begin{prop}
  \label{prop:countableimplieslem}
  If every subobject of~$\INN$ is countable and Markov's principle holds, then Excluded
  middle holds as well.
\end{prop}

\begin{proof}
  Let us show that the stated assumptions imply that every truth value $p$ is
  $\lnot\lnot$-stable. From an enumeration of the subobject $\set{n \in \INN \such p}$,
  which exists by assumption, we may construct $f : \INN \to \set{0, 1}$ such that~$p$ is
  equivalent to $\some{n \in \INN} f(n) = 1$. Markov's principle says that such a
  statement is $\lnot\lnot$-stable.
\end{proof}

\begin{cor}
  The statement ``not every subobject of~$\INN$ is countable'' is
  valid in~$\RT$.
\end{cor}

\begin{proof}
  A realizability topos validates Markov's principle and the negation
  of the law of excluded middle.
\end{proof}

\section{Intrinsic $T_0$-spaces are separable in function realizability}

We now generalize our result to \emph{synthetic topology}. We refer the reader
to~\cite{bauer12:_metric,lesnik10:_synth_topol_const_metric_spaces,escardo04:_synth_topol_data_types_class_spaces}
for comprehensive accounts of synthetic topology, and recall only those concepts
that are needed for our results.
The \emph{Rosolini dominance}~\cite{rosolini86:_contin_effec_topoi} is the
subobject of $\Omega$, defined by
\begin{equation*}
  \ISS =
  \set{ p : \Omega \such \some{\alpha : \set{0,1}^\INN} p \liff (\some{n \in \INN} \alpha_n = 1) }.
\end{equation*}
The object $\ISS$ is modest and can be thought of as an analogue of the
\emph{Sierpiński space}, which classifies open subspaces in the category of
topological spaces and continuous maps. This observation is the starting point
of synthetic topology, where the exponential $\ISS^X$ is taken to be the
\emph{intrinsic topology} of an object~$X$. Thus a subobject of $X$ is
\emph{intrinsically open} when its characteristic map $X \to \Omega$ factors
through the inclusion $\ISS \into \Omega$. We call the elements of $\ISS$
the \emph{open} truth values.

With regards to metric spaces, a fundamental question is how the metric and
intrinsic topologies relate. Because the relation~$<$ on $\IRR$ is intrinsically
open~\cite[Prop.~2.1]{bauer12:_metric}, every metric open ball is intrinsically
open. We note that the converse holds for $\lnot\lnot$-stable subobjects
of~$\IBB$.

\begin{prop}
  \label{prop:notnot-stable-metric-intrinsic}%
  Every intrinsically open subspace of a $\lnot\lnot$-stable subobject
  $S \into \IBB$ is a union of open balls.
\end{prop}

\begin{proof}
  The object~$\ISS$ is a quotient of $\IBB$ by the map
  $\beta \mapsto (\some{n}{\beta(n) = 1})$.
  Consider an intrinsically open subset $U \subseteq S$, with characteristic map
  $u : S \to \ISS$. By Extended function choice $u$ factors through the
  quotient $\IBB \to \ISS$ as a map $\overline{u} : S \to \IBB$. Moreover, by
  the Extended continuity principle the map $\overline{u}$ is continuous.
  Suppose $\alpha \in S$ and $u(\alpha) = \top$. There is $n \in \INN$ such that
  $\overline{u}(\alpha)(n) = 1$. By continuity of~$\overline{u}$ there is an
  open ball~$B(\alpha, r)$ centered at~$\alpha$ such that, for all
  $\beta \in B(\alpha, r)$, we have $\overline{u}(\beta)(n) = 1$, therefore
  $\alpha \in B(\alpha, r) \subseteq U$, as required.
\end{proof}

Many standard topological notions may be formulated in synthetic topology. For
instance, an object~$X$ is an \emph{(intrinsic) $T_0$-space} when the transpose
$X \to \ISS^{\ISS^X}$ of the evaluation map $\ISS^X \times X \to \ISS$
is a monomorphism, i.e., when two points in~$X$ are equal if they have the same
open neighborhoods. Metric spaces are $T_0$-spaces because metric open balls are
intrinsically open.

In classical topology arbitrary unions of opens are open, but this is not so in
synthetic topology. We say that $I$ is \emph{overt} if $\ISS$ is closed under
$I$-indexed unions. In logical form overtness says that for every
$u : I \to \ISS$ the truth value $\some{i \in I} u(i)$ is open. Because
$\ISS$ is a lattice, Kuratowski-finite objects are overt. The natural numbers
$\INN$ are overt as well, by an application of Countable choice.

In general overtness transfers along \emph{(intrinsically) dense} maps, which
are maps $f : X \to Y$ such that, for all $u : Y \to \ISS$,
\begin{equation*}
  (\some{y \in Y} u(y)) \iff (\some{x \in X} u(f(x))).
\end{equation*}
Clearly, if~$X$ is overt and $f : X \to Y$ is dense then $Y$ is overt.

Say that $X$ is \emph{(intrinsically) separable} when there exists a dense map
$\INN \to X$. Because $\INN$ is overt, all intrinsically separable objects are
overt.

\begin{lem}
  \label{lem:intrinsic-subB-separable}
  Every $\lnot\lnot$-stable subobject of $\IBB$ is intrinsically separable.
\end{lem}

\begin{proof}
  By Proposition~\ref{prop:notnot-stable-metric-intrinsic}, the intrinsic and
  metric topologies of a $\lnot\lnot$-stable subobject $S \into \IBB$ agree,
  hence so do both kinds of separability. In
  Proposition~\ref{prop:subB-separable} we showed that~$S$ is metrically
  separable.
\end{proof}

As noted earlier, every metric space is a $T_0$-space, and intrinsic
separability implies metric separability. Thus, the following is a
generalization of Theorem~\ref{thm:metric-separable}.

\begin{thm}
  \label{thm:t0-separable}
  In $\RT$, every intrinsically $T_0$-space is intrinsically separable.
\end{thm}

\begin{proof}
  Suppose $X$ is a $T_0$-space. Then it is a subobject of the modest object
  $\ISS^{\ISS^X}$, and so it is also modest. Hence there exists a surjection
  $q : S \onto X$ from a $\lnot \lnot$-stable subobject $S \into \IBB$. By
  Lemma~\ref{lem:intrinsic-subB-separable} there is a dense map
  $f : \INN \to S$. But a dense map followed by a surjection is dense, and so
  $q \circ f$ witnesses intrinsic separability of~$X$.
\end{proof}

Finally, a generalization of Theorem~\ref{thm:decidable-countable} is readily
available. An object $X$ is \emph{(intrinsically) discrete} when the equality
relation on~$X$ is open, or equivalently, when every singleton in~$X$ is open.

\begin{thm}
  \label{thm:discrete-countable}
  In~$\RT$, every intrinsically discrete space is countable.
\end{thm}

\begin{proof}
  An intrinsically discrete space $X$ is a $T_0$-space because in~$X$ singletons
  are open. By Theorem~\ref{thm:t0-separable} there is a dense map
  $f : \INN \to X$, which must be a surjection because, again, singletons are
  open.
\end{proof}

\section{Conclusion}
\label{sec:conclusion}

In view of our results, it is natural to wonder what goes wrong
intuitionistically with the classical non-separable spaces, such as
$\ell^\infty$ and $L^\infty[0,1]$. Do they somehow become separable? One of
several things can happen. In $\RT$ every map $[0,1] \to \IRR$ is uniformly
continuous so that $L^\infty[0,1]$ is just the space of uniformly continuous
maps with the supremum norm, which of course is separable. On the other hand,
$\ell^\infty$ cannot even be constructed, at least not using the classical
definition of the norm
\begin{equation*}
  \|a\|_\infty = \sup_{n \in \INN} |a_n|.
\end{equation*}
For a bounded $a : \INN \to \IRR$ the supremum need not exist, so that further
restrictions on~$a$ are required. The most generous attempt would collect into
$\ell^\infty$ all the sequences for which $\|a\|_\infty$ exists\,---\,but doing
so would break the vector space structure, and consequently the definition of
the metric. We could also observe that $\|a\|_\infty$ is a well-defined
\emph{lower} real number, i.e., we may give it as a lower Dedekind cut. How much
of the mathematics of $\ell^\infty$, and similar spaces, one can recover this
way was studied by Fred Richman~\cite{Richman:1998},

Lastly, we remark that the results are quite closely tied to~$\RT$ because of
the non-computable nature of the Excluded middle for predicates on~$\INN$ from
Proposition~\ref{prop:rt-principles} (apply the principle to the statement
``the $n$-th Turing machine halts'' to obtain the halting oracle).
A close cousin of $\RT$ is the Kleene-Vesley topos, which is defined
as the relative realizability topos on the partial combinatory subalgebra
of~$\Ktwo$ consisting of the computable functions~\cite[\S
4.5]{oosten08:_realiz}.
Because in this topos all statements must be realized by computable maps, the
last part of the proof of Proposition~\ref{prop:rt-principles} fails. Indeed, an
uncountable object with decidable equality is readily available, just take a
subset of~$\NN$ which is not computably enumerable, and therefore not countable
internally to the topos.

\bibliographystyle{plain}
\bibliography{all-spaces-separable}

\end{document}